\documentclass[a4paper]{amsart}
\usepackage[latin1]{inputenc}
\usepackage{amsmath,amsthm, amscd, amssymb, amsfonts}
\usepackage[dvips]{graphicx}
\usepackage[all]{xy}
\usepackage{leftidx}

\numberwithin{equation}{section}

\theoremstyle{plain}

\newtheorem{theorem}{Theorem}[section]
\newtheorem{corollary}[theorem]{Corollary}
\newtheorem{proposition}[theorem]{Proposition}
\newtheorem{lemma}[theorem]{Lemma}
\newtheorem{question}[theorem]{Question}
\newtheorem{lemma-definition}[theorem]{Lemma-Definition}

\theoremstyle{definition}
\newtheorem{definition}[theorem]{Definition}

\newtheorem{example}[theorem]{Example}

\theoremstyle{remark}
\newtheorem{remark}[theorem]{Remark}

\newcommand{\A}{{\mathbb A}}
\newcommand{\Aa}{{\mathfrak{A}}}
\newcommand{\Bb}{{\mathfrak{B}}}
\newcommand{\Kc}{{\mathfrak{K}}}
\newcommand{\Ll}{{\mathfrak{L}}}
\newcommand{\Cc}{{\mathfrak{C}}}
\newcommand{\Uu}{{\mathfrak{U}}}

\newcommand{\Hh}{\mathfrak{H}}
\newcommand\id{\operatorname{id}}

\newcommand\Fr{\operatorname{Fr}}
\newcommand\co{\operatorname{co}}

\newcommand\Hom{\operatorname{Hom}}
\newcommand\Rep{\operatorname{Rep}}

\newcommand\vect{\operatorname{vect}}
\newcommand\Vect{\operatorname{Vect}}

\begin{document}

\title[Jordan-H\" older theorem for finite dimensional Hopf algebras]{Jordan-H\"
older theorem for finite dimensional Hopf algebras}
\author{Sonia Natale}
\address{Facultad de Matem\'atica, Astronom\'\i a y F\'\i sica,
Universidad Nacional de C\'ordoba, CIEM -- CONICET, (5000) Ciudad
Universitaria, C\'ordoba, Argentina}
\email{natale@famaf.unc.edu.ar \newline \indent \emph{URL:}\/
http://www.famaf.unc.edu.ar/$\sim$natale}
\thanks{This work was partially supported by CONICET, Secyt (UNC) and the
Alexander von Humboldt Foundation}
\date{\today}

\begin{abstract} We show that a Jordan-H\" older theorem holds for appropriately
defined composition
series of finite dimensional Hopf algebras.  This answers an open question of N.
Andruskiewitsch. In the course of our proof we establish analogues of the
Noether isomorphism theorems
of group theory for arbitrary Hopf algebras under certain faithful
(co)flatness assumptions. As an application, we prove an analogue of
Zassenhaus' butterfly lemma for finite dimensional Hopf algebras. We then use
these results to show that a Jordan-H\" older theorem holds as well for lower
and upper composition series, even though the factors of such series may be not
simple as Hopf algebras.
\end{abstract}

\maketitle

\section{Introduction}

Let $G$ be a group. Composition series of $G$, when they exist, provide a tool
to decompose the group $G$ into a collection of simple groups: the composition
factors of the series. Regarding this notion, a basic fundamental  result in
group theory is the Jordan-H\" older theorem, which asserts that the composition
factors of a group $G$ are in fact determined by  $G$,
independently of the choice of the composition series.

Let $k$ be a field.   Hopf algebra
extensions play an important r\^ ole in the classification problem of
finite dimensional Hopf algebras over $k$.
Indeed, suppose that $H$ is a finite dimensional Hopf algebra  over $k$ and $A$
is a normal  Hopf subalgebra, that is, a Hopf subalgebra stable under the
adjoint actions of $H$. Then the ideal $HA^+$, generated by the augmentation
ideal $A^+$ of $A$, is a Hopf ideal of $H$ and therefore $B = H/HA^+$
is a quotient Hopf algebra. In this way, $A$ gives rise canonically to an
exact sequence of Hopf algebras (see Section \ref{dos}):
$$k \longrightarrow A \longrightarrow H \longrightarrow B
\longrightarrow
k.$$  

As a consequence of the Nichols-Zoeller freeness theorem \cite{NZ}, any
exact sequence of finite dimensional Hopf algebras is \emph{cleft}, that is, it
admits a convolution invertible $B$-colinear and $A$-linear section
$B \to H$. This implies that $H$ can be recovered from $A$ and
$B$ plus some extra cohomological data. More precisely, $H$ isomorphic
as a Hopf algebra to a bicrossed product $A\# B$ with respect to 
suitable compatible data. See \cite{schneider-nb}, \cite{AD}.  

These considerations motivated the question of deciding if an analogue of the
Jordan-H\" older theorem of group theory does hold in the context of finite
dimensional Hopf
algebras, which was raised by N. Andruskiewitsch in \cite[Question
2.1]{andrus}. 

In this paper we show that this question has an affirmative answer. The 
following definition is proposed in \cite{AM}. Recall that 
a Hopf algebra is called \emph{simple} if it contains no proper nontrivial
normal Hopf subalgebra.

\begin{definition}\label{cs-invariant} A \emph{composition series} of $H$ is a
sequence of finite dimensional simple Hopf algebras $\Hh_1, \dots, \Hh_n$
defined recursively as follows: 

$\bullet$ If $H$ is simple, we let $n = 1$ and $\Hh_1 = H$.

$\bullet$ If $k \subsetneq A \subsetneq H$ is a normal Hopf subalgebra, and
$\Aa_1,
\dots, \Aa_m$, $\Bb_1, \dots,$ $\Bb_l$, are composition series of $A$ and $B =
H/HA^+$,
respectively, then we let $n = m+l$ and 
$$\Hh_i = \Aa_i, \quad \textrm{if } 1\leq i \leq m, \quad \Hh_i = \Bb_{i-m},
\quad \textrm{if } m < i \leq m+l.$$
The Hopf algebras $\Hh_1, \dots, \Hh_n$ will be called the \emph{factors} of
the series.   The number $n$ will be called the \emph{length} of the series.
\end{definition}

Every finite dimensional Hopf algebra admits a composition series. The next
theorem is one of the main results of the paper:

\begin{theorem}\label{jh-invariant} \emph{(Jordan-H\" older theorem for finite
dimensional Hopf algebras.)} 
Let  $\Hh_1, \dots, \Hh_n$ and $\Hh'_1, \dots, \Hh'_m$ be two composition series
of $H$.
Then there exists a bijection $f: \{1, \dots, n\} \to
\{1, \dots, m\}$ such that $\Hh_i \cong \Hh'_{f(i)}$ as Hopf algebras. 
\end{theorem}

Theorem \ref{jh-invariant} is proved in Section \ref{invariant}. 
Its proof relies on appropriate analogues
of the Noether isomorphism theorems of group theory that we establish, more
generally, in the context of arbitrary Hopf algebras under suitable faithful
(co)flatness
assumptions (Theorems \ref{1st-it}, \ref{2nd-hopf} and \ref{3d-it}).  

As a consequence of Theorem \ref{jh-invariant} the \emph{composition factors}
and the \emph{length} of
$H$, defined, respectively, as the factors and the length of any composition
series, are
well-defined invariants of $H$.
By definition, the composition factors of $H$ are simple Hopf
algebras. We prove some basic properties of these invariants, for instance, we
show that 
they are additive with respect to exact sequences of Hopf algebras, and
compatible with duality.

 We also apply the isomorphism theorems to prove an analogue of the
Jordan-H\" older theorem for lower and upper composition series of $H$; these
are lower (respectively, upper) subnormal series which do not
admit a proper refinement. See Definition \ref{def-lss} and Theorem \ref{jh}.
This
allows us to introduce the lower and upper composition factors of $H$ and its
lower and upper lengths, which are also well-defined invariants of $H$. We
remark that, in
contrast with the case of the composition factors, the lower or upper
composition factors are not necessarily simple as Hopf algebras (see Example
\ref{dual-group}). This motivates the question of deciding if there is an
intrinsic characterization of the Hopf algebras that can arise as lower
composition factors (see Question \ref{structure-factors}). Our
proof of Theorem \ref{jh}
follows the lines of the classical proof of the Jordan-H\" older theorem in
group theory. In particular, we prove analogues of the Zassenhaus' butterfly
lemma (Theorem \ref{butterfly}) and the Schreier's refinement theorem (Theorem
\ref{schreier}) for finite dimensional Hopf algebras.

We study some properties of lower and upper composition factors and their
relation with the composition factors. Unlike for the case of the length, the
lower and upper lengths are not additive with respect to exact sequences and
they are not invariant under duality in general. Nevertheless, we show that if
the
lower (respectively, upper) composition factors are simple Hopf algebras, then
they coincide, up to permutations, with the composition factors (see  
Proposition
\ref{simple-factors}).
We discuss some families of examples that include group algebras and their
duals, abelian extensions and Frobenius-Lusztig kernels.

We point out that neither the composition factors nor the upper or
lower composition factors of a finite dimensional Hopf algebra $H$ are
categorical invariants of $H$. In other words, they are not invariant under
twisting deformations of $H$. In fact, it was shown in \cite{GN} that there
exists a
(semisimple) Hopf
algebra such that $H$ is simple as a Hopf algebra and $H$ is twist equivalent to
the group algebra of a solvable group $G$. In particular, the categories of
finite dimensional representations of $H$ and $G$ are equivalent fusion
categories. Thus not even the length nor the lower or upper lengths are
invariant under twisting deformations. 
We discuss series of normal (right) coideal subalgebras, instead of normal Hopf
subalgebras, in Subsection \ref{ncs}. These kind of series are of a categorical
nature but, as it turns out, they fail
to enjoy a Jordan-H\" older theorem. 

The paper is organized as follows. In Section \ref{dos} we recall some
definitions and facts related to normality and exact sequences. In Section
\ref{tres} we prove the isomorphism theorems and discuss some consequences,
including the butterfly lemma. Theorem \ref{jh-invariant}
is proved in Section \ref{invariant}; several properties and examples of
composition series and its related invariants are also studied in this section.
In Section  \ref{cuatro} we study lower and upper composition series and give a
proof of Theorem \ref{jh}. 

Unless explicitly mentioned, $k$ will denote an arbitrary field. The
notation $\Hom$, $\otimes$, etc.,  will mean, respectively, $\Hom_k$,
$\otimes_k$, etc. 

\subsection*{Acknowledgement} The author thanks N. Andruskiewitsch for
suggesting Definition \ref{cs-invariant} and for his comments on a previous
version
of this paper. She also thanks the Humboldt Foundation, C. Schweigert and the
Mathematics Department of the University of Hamburg, where this work was done,
for the kind
hospitality.

\section{Preliminaries}\label{dos}

Let $H$ be a  Hopf algebra over $k$. In this section we recall some facts
about normal Hopf
subalgebras, normal Hopf algebra maps, and exact sequences of Hopf algebras. Our
references for
these topics are \cite{AD}, \cite{schneider}, \cite{takeuchi}. 

\subsection{Normal Hopf subalgebras}\label{nhs} The left and right adjoint
actions of $H$ on itself are
defined, respectively, in the form:
$$h. a = h_{(1)}a\mathcal S(h_{(2)}), \qquad a.h= \mathcal S(h_{(1)})ah_{(2)},
\qquad a, h \in H.$$

Let $K, L$ be subspaces of $H$. We shall say that $L$ \emph{right normalizes}
$K$ (respectively, \emph{left normalizes}
$K$) if $K$ is stable under the right (respectively, left) adjoint action of
$L$.
A subspace $K \subseteq H$ is
called  \emph{left normal} (respectively, right normal) if it is stable under
the left adjoint
action of $H$ (respectively, under the right adjoint action of $H$). 
$K$ is called \emph{normal} if it is stable under both adjoint actions.

Suppose $K$ is a Hopf subalgebra of $H$. If $H$ is left faithfully flat over
$K$, then $K$ is right normal if and only if
$K^+H \subseteq HK^+$; if $H$ is right faithfully flat over $K$, then $K$ is
left normal if only if $HK^+ \subseteq K^+H$ \cite[Theorem 4.4 (a)]{takeuchi}.

\begin{remark}\label{normal-ba} (i) Suppose $H$ has a
bijective antipode. If $K \subseteq H$ is a  Hopf
subalgebra, then $H$ is left faithfully flat over $K$ if and only if it is right
faithfully flat over $K$. In this case the antipode of $K$ is also bijective
\cite[Proposition 3.3
(b)]{takeuchi}.

 (ii) Suppose that $A, B$ are Hopf subalgebras of $H$
with bijective
antipodes (which is always the case if $H$ is finite dimensional).   Then $A$
right normalizes $B$ if and only if it left normalizes $B$. If this is the case,
then $AB = BA$.
In particular, if the antipode of $H$ is bijective and $K \subseteq H$ is a Hopf
subalgebra with bijective antipode, then $K$ is left normal in $H$ if and only
if it is right normal in $H$.
When $H$ is faithfully flat over $K$, these conditions are also equivalent
to  $HK^+ = K^+H$.
\end{remark}

\subsection{Normal Hopf algebra maps} The left and right
adjoint coactions of $H$ on itself are defined, respectively, as $h \mapsto
h_{(1)}\mathcal S(h_{(3)}) \otimes h_{(2)}$, and $h \mapsto
h_{(2)} \otimes h_{(1)}\mathcal S(h_{(3)})$, $h \in H$.

Let $\pi: H \to \overline H$ be a Hopf
algebra map. The map $\pi$ will be called \emph{left normal}  
(respectively, \emph{right normal}) 
if the kernel $I$ of $\pi$ is a subcomodule for the left (respectively, right)
adjoint coaction of $H$, and it will be called \emph{normal} if it is both left
and right normal.

Let ${}^{\co \pi}H$ and $H^{\co \pi}$ be the subalgebras of 
 $H$ defined, respectively, by
\begin{equation*}{}^{\co \pi}H  = \{ h \in   H:\, (\pi \otimes \id)\Delta(h) = 1
\otimes h\}, \quad 
H^{\co \pi}  = \{ h \in   H:\, (\id \otimes \pi)\Delta(h) = h \otimes
1\}.\end{equation*}
The subalgebra
${}^{\co \pi}H$ is always a right normal right coideal subalgebra of $H$.
Similarly, $H^{\co \pi}$ is a left normal left coideal subalgebra of $H$. 
When the Hopf algebra map $\pi$ is clear from the context, we shall also use the
notation ${}^{\co \overline H}H$ and $H^{\co \overline H}$ to indicate the
subalgebras ${}^{\co \pi}H$ and $H^{\co \pi}$, respectively.

If $H$ is left faithfully coflat over $\overline H$, then 
$\pi$ is right normal if and only if $H^{\co \pi} \subseteq {}^{\co \pi}H$, and
if $H$ is right faithfully coflat over $\overline H$, then $\pi$ is left normal
if and only if ${}^{\co \pi}H \subseteq H^{\co \pi}$.
If $\pi: H \to \overline H$ is a left normal Hopf algebra map, then ${}^{\co
\pi}H$ is a right normal Hopf subalgebra of $H$. See \cite[Section 4]{takeuchi}.

\begin{remark}\label{obs-ad} (i) Suppose that the antipode of $H$ is bijective.
Then $H$ is left faithfully coflat over  $\overline H$ if and only if it is
right faithfully coflat. In this case, 
$\pi$ is left normal   if and only if it is right normal, if and only if $H^{\co
\pi}
= {}^{\co \pi}H$. Moreover, the antipode
of $\overline H$ is also bijective \cite[Proposition 3.1 and Corollary 4.5
(b)]{takeuchi}.

(ii) Suppose $H$ has a bijective antipode.
Let $K$ be a normal Hopf subalgebra of $H$. It follows from \cite[Theorem
3.2]{takeuchi} that $H$ is faithfully flat over $K$ if and only if it is
faithfully coflat over $H/HK^+$. 
If this is the case, then the antipodes of $K$ and
$H/HK^+$ are also bijective.  
\end{remark}

\begin{remark}\label{suff-iii} (i) When $H$ is a commutative Hopf algebra, $H$
is
faithfully flat over any Hopf subalgebra \cite{takeuchi-72}.
If the coradical of $H$ is cocommutative,  then $H$ is (left and right)
faithfully coflat over any quotient $H$-module coalgebra and (left and right)
faithfully flat over any Hopf subalgebra \cite{masuoka-c}.
If $H$ is finite dimensional, the 
Nichols-Zoeller theorem \cite{NZ} implies that $H$ is free over any Hopf
subalgebra. 
Observe that a Hopf algebra $H$ in any of the above classes has
a bijective antipode.

(ii) Every Hopf algebra $H$ is left and right
faithfully flat (in fact free) over its finite dimensional normal Hopf
subalgebras and left and right faithfully coflat over its finite dimensional
normal quotient Hopf algebras \cite[2.1]{schneider}.
\end{remark}

\subsection{Exact sequences of Hopf algebras}

An \emph{exact sequence of Hopf algebras} is a sequence of Hopf algebra maps 
\begin{equation}\label{exacta} k \longrightarrow H' \overset{i}\longrightarrow H
\overset{\pi}\longrightarrow H'' \longrightarrow k,
\end{equation}
satisfying the following conditions:
$$\textrm{(a) } i \textrm{ is injective and } \pi \textrm{  is surjective
},\quad
\textrm{(b) } \ker \pi = Hi(H')^+, \quad
\textrm{(c) } i(H') = {}^{\co \pi}\!H.$$
Note that either (b) or (c) imply that $\pi i = \epsilon 1$. If $H$ is
faithfully
flat over  $H'$, then (a) and (b) imply (c).  Dually, if $H$ is
faithfully coflat over $H''$, then (a) and (c)
imply (b).

Let $K \subseteq H$ be a normal Hopf subalgebra. Then $HK^+ = K^+H$
is a Hopf ideal of $H$ and the canonical map $H \to H/HK^+$ is a Hopf algebra
map. Hence, if $H$ is faithfully flat over $K$, then there is
an exact sequence of Hopf algebras $k \longrightarrow K \longrightarrow H
\longrightarrow H/HK^+ \longrightarrow
k$. Similarly, if $\pi: H \to \overline H$ is a normal quotient Hopf
algebra, then ${}^{\co \pi}H = H^{\co \pi}$ is a Hopf subalgebra and if $H$ is
faithfully coflat over $\overline H$,
there is an exact sequence of Hopf algebras $k \longrightarrow {}^{\co \pi}H
\longrightarrow H \longrightarrow \overline H \longrightarrow k$.

 Assume that $H$ is finite dimensional. Then any exact sequence \eqref{exacta}
is cleft. 
In particular,  $H \cong H' \otimes H''$ as left
$H'$-modules and right $H''$-comodules \cite{schneider-nb}, and therefore $\dim
H = \dim H' \dim H''$.

Observe that $H$ is simple if and only
if it admits no proper normal quotient Hopf algebra.
Furthermore, a sequence of Hopf algebra maps $k \longrightarrow H'
\overset{i}\longrightarrow H
\overset{\pi}\longrightarrow H'' \longrightarrow k$ is an exact sequence if and
only if 
$k \longrightarrow (H'')^* \overset{\pi^*}\longrightarrow H^*
\overset{i^*}\longrightarrow (H')^* \longrightarrow k$ is an exact sequence.
Therefore $H$ is simple if and only if $H^*$ is simple. 

\begin{remark}\label{ss-coss} Let \eqref{exacta} be an exact sequence of finite
dimensional
Hopf algebras. By cleftness,  $H$ is isomorphic to a bicrossed
product $H'\# H''$.
This implies that $H$ is semisimple if and only if $H'$ and $H''$ are
semisimple. Indeed, if $H$ is semisimple then every quotient Hopf algebra and
every Hopf subalgebra of $H$ are semisimple; for the converse, we use that $H$
is isomorphic as an algebra to a crossed product and  \cite[Theorem 2.6]{BM}. 
Dualizing the exact sequence \eqref{exacta} we also get that
$H$ is cosemisimple if and only if $H'$ and $H''$ are cosemisimple.
\end{remark}

\section{Isomorphism theorems for Hopf algebras}\label{tres}

In this section we prove analogues of the Noether isomorphism theorems. 
We refer the reader to \cite{masuoka-qt}, \cite{takeuchi}, for a detailed
exposition on the quotient theory of Hopf algebras.
Our proofs rely on the following result of M. Takeuchi:

\begin{theorem}\label{ker-copi-1}\emph{(\cite[Theorems 1 and 2]{takeuchi-HM}.)}
Let $H$ be a Hopf algebra. Let $K$ be a right coideal subalgebra of $H$ and let 
$\pi: H \to \overline H$ be a quotient left $H$-module coalgebra.   Then the
following hold:

(i) If $H$ is left faithfully flat over $K$, then $H$ is left faithfully
coflat over $H/HK^+$ and
$K = {}^{\co H/HK^+\!}\!H$.

(ii) If
$H$ is right faithfully coflat over $\overline H$, then $H$ is right 
faithfully flat over $^{\co \pi}H$ and $\ker \pi = H(^{\co
\pi}H)^+$. \qed 
\end{theorem}

\subsection{First isomorphism theorem}\label{1st-iso} In view of
\cite[Proposition 1.4 (a)]{takeuchi}, if $K \subseteq H$ is a right normal right
coideal subalgebra, then $HK^+$ is a Hopf ideal of $H$. Therefore the quotient
$H/HK^+$ is a Hopf algebra
and the canonical map $\pi_K: H \to H/HK^+$ is a Hopf algebra map. Theorem
\ref{ker-copi-1} can be regarded as an analogue of the first isomorphism theorem
of group theory:

\begin{theorem}\label{1st-it}\emph{(First isomorphism theorem for Hopf
algebras.)} Let $H$ be a Hopf algebra and let $\pi: H \to \overline H$ be a
surjective
 Hopf algebra map. Suppose that $H$ is right $\overline H$-faithfully
coflat. Then $\pi$ induces an isomorphism of Hopf algebras
$H/HK^+ \cong \overline H$,
where $K = {}^{\co \pi}H$. \qed 
\end{theorem}

As a consequence of Theorem \ref{ker-copi-1} we also obtain:
\begin{corollary}\label{p-thm}  Let $K \subseteq H$ be a right
normal right coideal subalgebra of $H$ and let $\pi: H \to \overline H$ be a
Hopf algebra map. Assume that $H$ is left faithfully flat over  $K$  and right
faithfully coflat over $\overline H$.
Then the following assertions are equivalent:

(i) $K \subseteq {}^{\co \pi}H$.

(ii) There is a unique Hopf algebra map $\overline\pi: H/HK^+ \to
\overline H$ such that $\pi = \overline \pi \pi_K$. \end{corollary}

\begin{proof} Assume (i). By Theorem \ref{ker-copi-1} (ii), we have that $\ker
\pi = H(^{\co \pi}H)^+$. Therefore $HK^+ \subseteq \ker \pi$ and there
is a unique algebra map $\overline \pi: H/HK^+ \to \overline H$ that
makes the diagram in (ii) commute. Since $\pi$ is a Hopf algebra map, then so
is $\overline \pi$. Hence (ii) holds.
Conversely, assume (ii). By  Theorem \ref{ker-copi-1} (i),  $K = {}^{\co
H/HK^+}H$. The relation $\pi = \overline \pi \pi_K$ implies $K = {}^{\co
H/HK^+}H
\subseteq {}^{\co \pi}H$. Thus (i) holds.
\end{proof}

\subsection{Second isomorphism theorem} A version of the next theorem
was shown, in a finite dimensional context, in \cite[(3.3.8)]{notes-ext}.

\begin{theorem}\label{2nd-hopf}\emph{(Second isomorphism theorem for Hopf
algebras.)} Let $H$ be a Hopf algebra and let $A, B$
be  Hopf subalgebras of $H$ such that $A$ right normalizes $B$. Then the
following hold:

(i) The product $AB$ is a Hopf subalgebra of $H$.

(ii) $A(A\cap B)^+$ is a Hopf ideal of $A$ and $AB^+$ is a
Hopf ideal of $AB$.

(iii) Suppose that $AB$ is left faithfully flat over $B$ and $A$
is right faithfully coflat over $AB/AB^+$. Then 
the inclusion $A \to
AB$ induces canonically an isomorphism of Hopf algebras
$A/A(A\cap B)^+ \cong AB/AB^+$.
\end{theorem}

\begin{proof} (i). It is clear that $AB$ is a subcoalgebra of $H$. Since $A$
right normalizes $B$, then, for all $b \in B$, $a\in A$, we have $ba =
a_{(1)}\mathcal S(a_{(2)})ba_{(3)} \in AB$. Thus $BA \subseteq AB$. Therefore
$AB$ is a subalgebra, hence a subbialgebra, of $H$. In addition, $\mathcal S(AB)
= \mathcal S(B) \mathcal S(A) \subseteq BA \subseteq AB$. Then $AB$ is a Hopf
subalgebra. 

 (ii). Since $A$ right normalizes $B$, then $A\cap B$ is a right normal
Hopf subalgebra of $A$ and $B$ is a right normal Hopf subalgebra of $AB$. Hence
$A(A\cap B)^+$ is a Hopf ideal of $A$ and $AB^+ = ABB^+$ is a Hopf
ideal of $AB$ \cite[Proposition 1.4 (a)]{takeuchi}.

 (iii). It follows from (ii) that $A/A(A\cap B)^+$ is a quotient Hopf
algebra of
$A$ and $AB/AB^+$ is a quotient Hopf algebra of $AB$. The composition of the
inclusion $A \to AB$ with the
canonical projection $\pi': AB \to AB/AB^+$ induces a Hopf algebra map $\pi':
A\to AB/AB^+$. Note that $\pi'$ is surjective; indeed,  for all $a \in A$, $b
\in B$, we
have that $\pi'(ab) = \pi'(a)\pi'(b) = \pi'(a)\epsilon(b) \in \pi'(A)$. Hence
$AB/AB^+ = \pi'(AB) = \pi'(A)$.

Since $AB$ is left faithfully flat over $B$,  Theorem
\ref{ker-copi-1} (i) implies that $AB$ is left faithfully coflat over $AB/AB^+$
and ${}^{\co \pi'}(AB) = {}^{\co AB/AB^+}\!(AB) = B$.

By assumption, $A$ is right faithfully coflat over $AB/AB^+$. Then Theorem
\ref{ker-copi-1} (ii) implies that $A$ is right faithfully flat over ${}^{\co
\pi'}\!A$ and the kernel of $\pi': A \to AB/AB^+$ coincides with $A (^{\co
\pi'}\!A)^+$.  Note that  ${}^{\co \pi'}\!A = A \cap {}^{\co \pi'}\!(AB) = A
\cap B$. 
From this we deduce that the kernel of $\pi': A \to AB/AB^+$ is $A(A\cap
B)^+$. Hence $\pi'$ induces an
isomorphism of Hopf algebras $A/A(A\cap B)^+ \to AB/AB^+$, as claimed.
\end{proof}

\begin{corollary} \label{dim-ab} Let $A, B$ be finite dimensional Hopf
subalgebras
of $H$ such that $A$ right normalizes $B$. Then $AB$ is a finite dimensional
Hopf subalgebra of $H$ and we have $\dim AB = \dfrac{\dim A \dim B}{\dim A \cap
B}$.
\end{corollary}

\begin{proof}By Theorem \ref{2nd-hopf} (iii), there is an
isomorphism of Hopf algebras $A/A(A\cap B)^+ \cong
AB/AB^+$ (c.f. Remark \ref{suff-iii}).  
This implies the corollary, since $\dim A/A(A\cap B)^+ = \dfrac{\dim A}{\dim
A\cap B}$ and $\dim AB/AB^+ = \dfrac{\dim AB}{\dim B}$.\end{proof}

\subsection{Third isomorphism theorem} A version of the following theorem was
established, under certain finiteness 
conditions, in \cite[(3.3.7)]{notes-ext}.
Recall that if $K\subseteq H$, then $\pi_K:H\to H/HK^+$ denotes 
the canonical map. 

\begin{theorem}\label{3d-it}\emph{(Third isomorphism theorem for Hopf
algebras.)} Let $A$ be a right normal Hopf subalgebra of $H$ and let $B$ be a
right normal right coideal subalgebra of $H$ such that $B \subseteq A$. 
Then $\pi_B(A)$ is a right normal Hopf subalgebra of $H/HB^+$
and the following hold:

(i) $\pi_A$ induces an isomorphism of Hopf algebras
$\dfrac{H/HB^+}{(H/HB^+) \pi_B(A)^+} \cong H/HA^+$.

(ii) Assume that $H$ is left faithfully flat over $B$ and $A$ is right
faithfully coflat over $\pi_B(A)$. Then $\pi_B$ induces an isomorphism of Hopf
algebras  $A/AB^+ \cong \pi_B(A)$. 
\end{theorem}

\begin{proof} Since $\pi_B$ is a surjective Hopf algebra map, then $\pi_B(A)$ is
a normal Hopf subalgebra of $H/HB^+$.
Since $HB^+ \subseteq HA^+ = \ker \pi_A$, there exists a unique surjective
Hopf algebra map $\overline{\pi_A}:H/HB^+\to H/HA^+$ such that $\pi_A =
\overline{\pi_A}\pi_B$.
Hence $\ker \overline{\pi_A} = \pi_B(\ker \pi_A) = \pi_B(HA^+) =
(H/HB^+)\pi_B(A)^+$.
Therefore (i) holds. 

Theorem \ref{ker-copi-1} (i) implies that ${}^{\co
\pi_B}H = B$, since $H$ is left faithfully flat over $B$. Hence ${}^{\co
\pi_B}A =  A\cap {}^{\co \pi_B}H = B$.  Therefore, since $A$ is right faithfully
coflat over $\pi_B(A)$, Theorem
\ref{1st-it} implies that $\pi_B$ induces an isomorphism of Hopf algebras 
$A/AB^+ \cong \pi_B(A)$. This proves part (ii) and finishes the proof of the
theorem.
\end{proof}

\begin{corollary}\label{adm-max} Let $B$ be a normal Hopf subalgebra of $H$ such
that $H$ is left faithfully flat over $B$.
Assume that the quotient Hopf algebra $H/HB^+$ is simple.  Then for every normal
Hopf subalgebra $A$ such that $B \subseteq A$ and $H$ is left faithfully flat
over $A$, we have $A = B$ or $A = H$.
\end{corollary}

\begin{proof}
Let $\pi_B: H\to H/HB^+$ denote the canonical map. By Theorem \ref{ker-copi-1}
(i), ${}^{\co \pi_B}H = B$, since $H$ is left faithfully flat over $B$. Since 
$\pi_B(A)$ is a normal Hopf subalgebra of $H/HB^+$ and  $H/HB^+$ is a simple
Hopf algebra  by assumption, then $\pi_B(A) = k$ or $\pi_B(A) = H/HB^+$.
If $\pi_B(A) = k$, then $\pi_B\vert_A = \epsilon_A$. Therefore $A
\subseteq {}^{\co \pi_B}H = B$ and $A = B$ in this case.
Suppose, on the other hand, that $\pi_B(A) = H/HB^+$. From Theorem \ref{3d-it},
we obtain that $\ker \overline{\pi_A} = (H/HB^+)\pi_B(A)^+ = (H/HB^+)^+$.
Hence
$H/HA^+
\cong (H/HB^+) / (H/HB^+) \cong k$. Since $H$ is left faithfully flat over $A$,
Theorem \ref{ker-copi-1} (i) implies that $A = H$. 
\end{proof}

\begin{example}\label{flker-max} Let $\mathfrak g$ be a finite dimensional
semisimple complex Lie algebra. Suppose $\ell$ is an odd
integer, relatively prime to 3 if $\mathfrak g$ has a component of type $G_2$,
and let $\varepsilon$ be a
primitive $\ell$-th root of unity in $\mathbb C$.
Let also $U(\mathfrak{g})$ be the universal enveloping algebra of $\mathfrak{g}$
over the cyclotomic field $k = \mathbb Q(\varepsilon)$ and let
$U_\varepsilon(\mathfrak g)$ be the quantum enveloping algebra
introduced by Lusztig \cite{luzstig}. Then $U_\varepsilon(\mathfrak g)$ is a
pointed Hopf algebra over $k$ with bijective antipode.

Consider the Frobenius homomorphism $\Fr: U_\varepsilon(\mathfrak g) \to
U(\mathfrak{g})$ \cite[8.10, 8.16]{luzstig}.  The Frobenius-Lusztig kernel of
$U_\varepsilon(\mathfrak g)$ is the finite dimensional Hopf subalgebra
$\mathfrak{u} = U_\varepsilon(\mathfrak g)^{\co \Fr}$ of
$U_\varepsilon(\mathfrak g)$.  $\Fr$
induces a cleft exact sequence of Hopf algebras
$k \longrightarrow \mathfrak u \longrightarrow U_\varepsilon(\mathfrak g)
\overset{\Fr}\longrightarrow U(\mathfrak{g}) \longrightarrow k$;
see \cite[Lemma 3.4.1 and Proposition 3.4.4]{notes-ext}. In particular, there
is  an isomorphism of Hopf algebras $U_\varepsilon(\mathfrak
g)/U_\varepsilon(\mathfrak g)\mathfrak{u}^+ \cong U(\mathfrak{g})$.

Suppose $\mathfrak{g}$ is a simple Lie algebra. Then
$U(\mathfrak{g})$ is a simple Hopf algebra (this can be seen as a consequence of
the Cartier-Kostant-Milnor-Moore theorem). 
Corollary  \ref{adm-max} implies that $\mathfrak{u} \subseteq U(\mathfrak{g})$
is a maximal faithfully flat normal inclusion of Hopf algebras.
\end{example}

\subsection{Summary} The following theorem summarizes the contents of
the previous subsections in the finite dimensional case.

\begin{theorem} Let $H$ be a finite dimensional Hopf algebra.
The following hold:

(i) Let $A, B$ be Hopf subalgebras of $H$ such that $A$
normalizes $B$. Then $B$ is a normal Hopf subalgebra of $AB$, $A\cap B$ is a
normal Hopf subalgebra of $A$ and there is an exact sequence of Hopf algebras
$k \longrightarrow A\cap B \longrightarrow A \longrightarrow
AB/AB^+\longrightarrow k$.

(ii) Let $A$ be a normal Hopf subalgebra and $B$ a right normal right coideal
subalgebra of $H$ such that $B \subseteq A$. 
Then there exists a unique Hopf algebra map $\pi:H/HB^+ \to H/HA^+$ such
that $\pi \pi_B = \pi_A$ and the map $\pi$ fits into an exact sequence of Hopf
algebras
$k \longrightarrow A/AB^+ \longrightarrow H/HB^+ \overset{\pi}\longrightarrow
H/HA^+\longrightarrow k$.
\end{theorem}

\begin{proof}  A finite
dimensional Hopf algebra is free over any Hopf
subalgebra and over any right coideal subalgebra \cite{NZ}, \cite{skryabin}. The
statement follows from Theorems \ref{1st-iso},
\ref{2nd-hopf} and \ref{3d-it}.
\end{proof}

\subsection{Zassenhaus' butterfly lemma} We assume in this subsection that $H$
is finite dimensional. 
As an application of Theorem \ref{3d-it} we obtain the following 
analogue of Zassenhaus' butterfly lemma of group theory:

\begin{theorem}\label{butterfly} Let $A,
B$ be Hopf subalgebras of $H$ and let $A', B'$ be normal Hopf subalgebras of $A$
and
$B$, respectively. Then the following hold:

(i) $A'(A\cap B')$ is a normal Hopf subalgebra of $A'(A\cap B)$.

(ii) $B'(A'\cap B)$ is a normal Hopf subalgebra of $B'(A\cap B)$.

(iii) $A'(A \cap B') \cap A \cap B = (A' \cap B) (A \cap B') = B'(A' \cap
B) \cap A \cap B$.

(iv) There is an isomorphism of Hopf algebras $$\dfrac{A'(A\cap
B)}{A'(A\cap B)(A'(A\cap B'))^+} \cong \dfrac{B'(A\cap B)}{B'(A\cap B)(B'(A'\cap
B))^+}.$$ 
\end{theorem}

\begin{proof} Note that, since $A'$ is normal in $A$, then $A'(A\cap B')$ and
$A'(A\cap B)$ are indeed Hopf subalgebras of $H$. Similarly, $B'(A'\cap B)$ and
$B'(A\cap B)$ are Hopf subalgebras of $H$, because $B'$ is normal in $B$. 
 Part (i) follows from the fact that $A'(A\cap B')$ is stable under
both  adjoint actions of $A'$ (since $A' \subseteq A'(A\cap B')$) and also under
both
adjoint
actions of $A\cap B$, because $A'$ is normal in $A$ and $B'$ is normal in $B$.
Part (ii) reduces to  (i),
interchanging the r\^ oles of $A$ and $B$.

We next prove (iii). It is clear that  
$(A' \cap B) (A \cap B') \subseteq A'(A \cap B') \cap A \cap B$.
By Corollary \ref{dim-ab}, we have 
\begin{equation}\label{dd}\dim (A' \cap B) (A \cap B') = \dfrac{\dim A' \cap B
\dim A \cap B'}{\dim A' \cap B'}.
\end{equation}
On the other side, using again Corollary \ref{dim-ab}, we compute 
\begin{align*}\dim A'(A \cap B') \cap A \cap B & = \dfrac{\dim A'(A \cap B')
\dim A \cap B}{\dim A'(A \cap B') (A \cap B)}  =  \dfrac{\dim A'(A \cap B')
\dim
A \cap B}{\dim A'(A \cap B)}\\ 
& = \dfrac{\dim A' \dim A\cap B' \dim A \cap B}{\dim A'\cap B' \dim A'(A\cap B)}
= \dfrac{\dim A\cap B' \dim A' \cap B}{\dim A'\cap B'}.
\end{align*}
Comparing this formula with \eqref{dd}, we find that $(A' \cap B) (A \cap B')$
and $A'(A \cap
B') \cap A \cap B$ have the same    (finite) dimension, and therefore they are
equal. By
symmetry, this also shows that $(A' \cap B) (A \cap B') = B'(A' \cap B) \cap A
\cap B$. Hence we get (iii). 

 To prove part (iv) we argue as follows. From Theorem \ref{2nd-hopf},
we obtain
\begin{align*}
\dfrac{A'(A\cap B)}{A'(A\cap B)(A'(A\cap B'))^+}  & = \dfrac{A' (A\cap B')
(A\cap
B)}{A'(A\cap B)(A'(A\cap B'))^+} \\ & \cong \dfrac{A\cap B}{(A\cap B) (A'(A\cap
B') \cap A \cap B)^+},\\
\textrm{and \quad } \dfrac{B'(A\cap B)}{B'(A\cap B)(B'(A'\cap B))^+} & =
\dfrac{B'
(A'\cap B) (A\cap B)}{B'(A\cap B)(B'(A'\cap B))^+} \\
& \cong \dfrac{A\cap B}{(A\cap B) (B'(A'\cap B) \cap A \cap B)^+}.
\end{align*}
Then part (iv) follows from (iii). This finishes the proof of the theorem.
\end{proof}

\section{Composition series of finite dimensional Hopf
algebras}\label{invariant}

Let $H$ be a finite dimensional Hopf algebra over $k$. Recall
 that  a composition series of $H$ is a
sequence of finite dimensional simple Hopf algebras $\Hh_1, \dots, \Hh_n$
defined as $\Hh_1 = H$ and $n = 1$, if $H$ is simple, and 
$n = m+l$, $\Hh_i = \Aa_i$, if $1\leq i \leq m$,
$\Hh_{i} = \Bb_{i-m}$, if $m < i \leq m+l$, whenever  $k \subsetneq A \subsetneq
H$ is a normal Hopf subalgebra, and $\Aa_1,
\dots, \Aa_m$, $\Bb_1, \dots, \Bb_l$, are composition series of $A$ and $B =
H/HA^+$,
respectively. See Definition \ref{cs-invariant}.

\begin{lemma} Every finite dimensional Hopf algebra admits a composition series.
\end{lemma}

\begin{proof} If $H$ is simple, then
$H_1 = H$ is a composition series of $H$. Otherwise, $H$
contains a normal Hopf subalgebra $k \subsetneq A \subsetneq H$. We have $\dim
A, \dim H/HA^+ < \dim H$, because $\dim H =\dim A \dim H/HA^+$. The lemma
follows by induction.
\end{proof}

As an application of the results in Section \ref{tres} we are now able to prove
the
Jordan-H\" older Theorem \ref{jh-invariant} for finite dimensional Hopf
algebras:

\begin{proof}[Proof of Theorem \ref{jh-invariant}] The proof is by induction on
the dimension of $H$. If $H$ is
simple, there is nothing to prove.
Assume that $H$ is not simple. Let $k \subsetneq A \subsetneq H$ be a normal
Hopf subalgebra, and let $\Aa_1,
\dots, \Aa_m$, $\Bb_1, \dots, \Bb_l$, be composition series of $A$ and $B =
H/HA^+$,
respectively. 

Suppose that there exists a normal Hopf subalgebra $K$ such that $A \subsetneq K
\subsetneq H$.
Let $\Kc_1, \dots, \Kc_r$, $\Ll_1, \dots, \Ll_s$, be composition series of $K$
and
$L = H/HK^+$,
respectively. Then $\Kc_1, \dots, \Kc_r, \Ll_1, \dots, \Ll_s$ is a composition
series of
$H$. Let $\Cc_1, \dots, \Cc_p$ be a composition series of $K/KA^+$.
Since $\dim K < \dim H$, it follows by induction that $r = m+p$ and the sequence
$\Kc_1, \dots,
\Kc_r$ is a permutation of the sequence $\Aa_1, \dots, \Aa_m, \Cc_1, \dots,
\Cc_p$.

By Theorem \ref{3d-it} (iii), there is an exact sequence of Hopf algebras $$k
\longrightarrow K/KA^+ \longrightarrow H/HA^+ \longrightarrow
H/HK^+\longrightarrow k.$$
We  also have $\dim H/HA^+ < \dim H$. Hence, by induction, $l = p+s$ and the
sequence $\Bb_1,
\dots, \Bb_l$ is a permutation of the
sequence $\Cc_1, \dots, \Cc_p, \Ll_1, \dots, \Ll_s$. In conclusion, $r + s = m +
l$ and
the sequence $\Kc_1,
\dots, \Kc_r, \Ll_1, \dots, \Ll_s$ is a permutation of the sequence $\Aa_1,
\dots, \Aa_m,
\Bb_1, \dots, \Bb_l$.

 Let now $k \subsetneq A' \subsetneq H$ be another normal
Hopf subalgebra, and let $\Aa'_1,
\dots, \Aa'_e$, $\Bb'_1, \dots, \Bb'_d$, be composition series of $A'$ and
$B' = H/H{A'}^+$,
respectively.  We want to show that $m+l = e+d$ and $\Aa_1,
\dots, \Aa_m$, $\Bb_1, \dots, \Bb_l$ is a permutation of $\Aa'_1,
\dots, \Aa'_e$, $\Bb'_1, \dots, \Bb'_d$. We may 
assume that $A \neq A'$. Furthermore, the first part of the proof implies that
we may also
assume that $A$ and $A'$ are maximal proper normal Hopf subalgebras. 
Since $AA'$ is a normal Hopf subalgebra of $H$ containing strictly $A$ and $A'$,
then $AA' = H$. In view of Theorem \ref{2nd-hopf}, $A\cap A'$ is a normal Hopf
subalgebra of $A$ and of $A'$ and there are 
isomorphisms $H/HA^+ \cong A'/A'(A'\cap A)^+$, $H/H{A'}^+
\cong A/A(A'\cap A)^+$. 

Let $\Uu_1, \dots, \Uu_f$ be a composition series of 
$A\cap A'$.
Since $\dim A, \dim A' < \dim H$, the inductive
assumption implies that $e = f+l$, $m = f+d$, $\Aa'_1,
\dots, \Aa'_e$ is a permutation of  $\Uu_1, \dots, \Uu_f, \Bb_1,\dots, \Bb_l$,
and  $\Aa_1,
\dots, \Aa_m$ is a permutation of $\Uu_1, \dots, \Uu_f$, $\Bb'_1,\dots,
\Bb'_d$. 
Hence $e+d = m+l$ and the sequence $\Aa'_1, \dots, \Aa'_e$, $\Bb'_1, \dots,
\Bb'_d$ is a
permutation of 
$\Aa_1, \dots, \Aa_m$, $\Bb_1, \dots, \Bb_l$, as claimed. This finishes the
proof of the
theorem.
\end{proof}

Theorem \ref{jh-invariant} permits to introduce the following invariants of a
finite
dimensional Hopf algebra:

\begin{definition} Let $H$ be a finite dimensional Hopf algebra and let 
$\Hh_1, \dots,$ $\Hh_n$
be a composition series of $H$. The simple Hopf algebras $\Hh_i$, $1\leq i\leq
n$, will be called the \emph{composition factors of $H$}. The number $n$ will be
called the \emph{length of $H$}.
\end{definition}

\begin{corollary}\label{additive} Suppose $k \longrightarrow A \longrightarrow H
\longrightarrow B
\longrightarrow k$ is any exact sequence of Hopf algebras.
Then the length of $H$ equals the sum of the lengths of $A$ and $B$.
\qed \end{corollary}

\begin{corollary}\label{dual-inv} Let $H$ be a finite dimensional Hopf algebra
and let $\Hh_1,
\dots,$ $\Hh_n$, be the compostion factors of $H$. 
Then the composition factors of $H^*$ are $\Hh_1^*, \dots, \Hh_n^*$. In
particular
the length of $H^*$ coincides with the length of $H$. 
\end{corollary}

\begin{proof} The proof is by induction on the dimension of $H$, using the fact
that a sequence 
 $k \longrightarrow A \longrightarrow H \longrightarrow B
\longrightarrow k$ is exact if and only if the sequence $k \longrightarrow B^*
\longrightarrow H^* \longrightarrow A^*
\longrightarrow k$ is exact.
\end{proof}

\begin{proposition}\label{ss-invfactors} Let $H$ be a finite dimensional Hopf
algebra. Then $H$ is
semisimple (respectively, cosemisimple) if and only if all its composition
factors are semisimple (respectively, cosemisimple).  \end{proposition}

\begin{proof} We shall prove the statement concerning semisimplicity. This
implies the statement for cosemisimplicity, in view of Corollary \ref{dual-inv}.
We may assume that  $H$ is not simple. Then there is an exact sequence $k
\longrightarrow A \longrightarrow H \longrightarrow B \longrightarrow k$, where
$\dim A, \dim B < \dim H$. In particular, $H$ is semisimple if and only if $A$
and
$B$ are semisimple (see Remark \ref{ss-coss}). By definition, the composition
factors of $H$ are
$\Aa_1, \dots, \Aa_m$, $\Bb_1, \dots, \Bb_l$, where $\Aa_1, \dots, \Aa_m$ are
the commposition factors of $A$ and  $\Bb_1, \dots, \Bb_l$ are the composition
factors of $B$.  The proposition follows by an 
inductive argument.
\end{proof}

\begin{example} Let $G$ be a finite group and let $H =kG$ be the group Hopf
algebra of $G$. The normal Hopf subalgebras of $H$ are exactly those group
algebras $kN$, where $N$ is a normal
subgroup of $G$ and $kG/kG(kN)^+ \cong k(G/N)$. Thus the composition factors of
$H$ are exactly the group
algebras of the composition factors of $G$. In view  of Corollary \ref{dual-inv}
the composition factors of the dual group algebra $k^G$ are the dual group
algebras of the composition factors of $G$.
\end{example}

\begin{example}\label{ddoble} Let $F$, $\Gamma$ be finite groups and let $H$ be
an abelian extension of $k^\Gamma$ by $kG$, that is, $H$ is a Hopf algebra
fitting into an exact sequence $k \longrightarrow k^\Gamma \longrightarrow H
\longrightarrow kF \longrightarrow k$ (see Example \ref{ej-ab}). Then the
composition
factors of $H$ are the
group algebras of the composition factors of $F$ and the dual group algebras of
the composition factors of $\Gamma$. 

As an example, consider the Drinfeld double $D(G)$ of a finite group $G$; $D(G)$
fits into an exact sequence $k \longrightarrow k^G \longrightarrow D(G)
\longrightarrow kG \longrightarrow k$.
Therefore if $G_1, \dots, G_n$ are  the composition
factors of $G$, then the composition factors of $D(G)$ are the
Hopf algebras  $k^{G_1},
\dots, k^{G_n}, kG_1, \dots, kG_n$. In particular, the length of $D(G)$ is twice
the length of $G$.
\end{example}

\section{Upper and lower composition series}\label{cuatro} 

Along this section, $H$ will be a finite dimensional Hopf algebra over $k$. The
following definition extends the notion of subnormal series of a group. 

\begin{definition}\label{def-lss} A
\emph{lower subnormal
series} of $H$ is
a series of Hopf subalgebras
\begin{equation}\label{lowerseries}k = H_{n}  \subseteq H_{n-1}
\subseteq  \dots
\subseteq H_1 \subseteq H_0 = H, \end{equation} with $H_{i+1}$
normal in $H_i$, for all $i$. The \emph{factors} of the series
\eqref{lowerseries} are the quotient Hopf algebras $\overline H_i =
H_i/H_iH_{i+1}^+$, $i = 0, \dots, n-1$. 

An \emph{upper
subnormal series} of $H$ is a series of surjective Hopf
algebra maps
\begin{equation}\label{upperseries}H = H_{(0)} \to H_{(1)}
\to \dots \to H_{(n)} = k,
\end{equation} such that $H_{(i+1)}$ is a normal quotient Hopf algebra of
$H_{(i)}$, for all $i = 0, \dots, n-1$. The \emph{factors} of
\eqref{upperseries} are the Hopf algebras $\underline H_i = {}^{\co
H_{(i+1)}}H_{(i)} \subseteq H_{(i)}$, $i = 0, \dots, n-1$. 
\end{definition}

Lower and upper subnormal series were introduced in \cite[Section 3]{MW} under
the names \emph{normal} lower and upper series, respectively.  

\begin{definition} A \emph{refinement} of 
\eqref{lowerseries} is a lower subnormal series
\begin{equation}\label{refi}k = H'_{m} \subseteq H'_{m-1} \subseteq \dots
\subseteq H'_1 \subseteq H'_0 = H,\end{equation}
such that for all  $0 < i < n$, there exists $0 < N_i < m$, with $N_1 <
N_2 < \dots < N_m$, and $H_i = H'_{N_i}$.  If  \eqref{refi} is a refinement of
\eqref{lowerseries} and it does not coincide
with \eqref{lowerseries}, we shall say that it is a  \emph{proper refinement}.

Two lower subnormal series $k = H_{n} \subseteq H_{n-1} \subseteq
\dots
\subseteq H_1 \subseteq H_0 = H$ and $k = H'_{m} \subseteq H'_{m-1} \subseteq
\dots
\subseteq H'_1 \subseteq H'_0 = H$ will be called \emph{equivalent} if there
exists a bijection $f: \{0, \dots, n\} \to \{0, \dots, m\}$ such that the
corresponding factors are isomorphic as Hopf algebras, that is, such that 
$H_{i}/H_{i}H_{i+1}^+ \cong H'_{f(i)}/H'_{f(i)}{H'_{f(i)+1}}^+$.

 A \emph{lower composition series} of $H$ is a \emph{strictly
decreasing} lower subnormal series which does not admit a proper refinement. In
other words, a lower normal series \eqref{lowerseries} is a composition series
if and only if $H_i$ is a maximal normal Hopf subalgebra of
$H_{i-1}$, for all $i = 1, \dots,  n$.
\end{definition}

Upper composition series can be defined similarly. It is clear that every finite
dimensional Hopf algebra $H$ admits lower and upper composition series. 

\begin{remark}\label{duality} It follows from \cite[Theorem 3.2]{MW} that lower 
subnormal series
of $H$ with factors $\overline H_i$, $i = 0, \dots, n-1$,
correspond to upper subnormal series of $H^*$ with factors $\underline{H^*}_i
\cong (\overline H_i)^*$, $i = 0, \dots, n-1$. Similarly, upper  subnormal
series of $H$ with factors $\underline H_i$, $i = 0, \dots, n-1$, correspond
to lower subnormal series of $H^*$ with factors $\overline{H^*}_i
\cong \underline H_i^*$, $i = 0, \dots, n-1$.
In particular, upper composition series of $H$ with factors $\underline H_i$
correspond to
lower composition series of $H^*$ with factors $\underline H_i^*$. \end{remark}

\begin{proposition}\label{simple-factors} Let $H$ be a finite dimensional Hopf
algebra. Then a lower (respectively, upper) subnormal series of $H$ all of whose
factors are simple Hopf algebras is a lower (respectively, upper)
composition series. The factors of such series coincide, up to permutations,
with the composition
factors of $H$. 
\end{proposition}

\begin{proof} Suppose \eqref{lowerseries} is a lower composition series with
simple factors. 
By Corollary \ref{adm-max}, \eqref{lowerseries} is a lower
composition series.  Let
$\overline H_i$, $i = 0, \dots, n-1$, be the lower composition factors of $H$.
Observe that $\overline H_i$, $i \geq 1$, are the lower composition
factors of $H_1$, and we may assume inductively that they coincide with its
composition factors. 
Since $H_1$ is normal in $H$ and $\overline H_0 = H/HH_1^+$ is simple, then the
composition factors of $H$ are, up to permutations, $\overline H_0$ and
$\overline H_i$, $i = 1, \dots, n-1$, as claimed.
\end{proof}

\begin{example}\label{cocommutative} Assume $H$ is a cocommutative Hopf
algebra. Then every quotient Hopf algebra of $H$ is normal. Suppose $B$ is a
normal Hopf subalgebra and $H/HB^+ \to Q$ is a
proper quotient Hopf algebra. Then ${}^{\co Q}H$ is a  normal Hopf
subalgebra of $H$ such that $B \subseteq {}^{\co Q}H$ and ${}^{\co Q}H \neq H$. 
This shows that the converse of Corollary \ref{adm-max} is true in this case. 
Hence a lower subnormal series of $H$ is a lower composition series
if and only if all its factors are simple Hopf algebras. By Proposition
\ref{simple-factors} the composition factors of $H$ coincide with  its lower
composition factors.
In particular, if $H$ is the group algebra of a finite group $G$, then
 the lower composition factors of $H$ are the group algebras of the composition
factors of $G$.
\end{example}

The next example shows that the converse of Proposition \ref{simple-factors} is
not true.

\begin{example}\label{dual-group} 
Let $G$ be a finite group and let $k^G$ be the
commutative Hopf algebra of $k$-valued functions on $G$. The Hopf subalgebras of
$k^G$ are of the form $k^{G/S}$, where $S$ is a normal subgroup of $G$, and
every Hopf subalgebra is normal. 

A lower composition series of $k^G$ corresponds
to a \emph{principal series} (or \emph{chief series}) of $G$, that is, a
strictly
increasing series of subgroups $S_0 = \{e\} \subseteq S_1 \subseteq \dots
\subseteq S_{n-1} \subseteq S_n = G$, where  $S_i$
is a normal subgroup of $G$, $0\leq i \leq n-1$, and there exists no normal
subgroup $S$ of $G$ such
that $S_i \subsetneq S \subsetneq S_{i+1}$ (see \cite[Section 1.3]{robinson}).

Recall that a \emph{characteristic subgroup} of a group $G$ is a subgroup stable
under the action of the automorphism group of $G$. The factors of a principal
series  (called the principal or chief factors of $G$)
are  \emph{characteristically simple} groups, that
is, they contain no proper characteristic subgroup.
It is a known fact that a finite group is characteristically simple if and only
if it is
isomorphic to a direct product of isomorphic simple groups (see for instance
\cite[3.3.15]{robinson}).
In particular, the lower composition factors of $k^G$ are not necessarily simple
Hopf algebras. 
\end{example}

The fact that a lower composition series can have factors which are not simple
motivates the following question:

\begin{question}\label{structure-factors} Can a lower (or upper) composition
series of a finite dimensional Hopf algebra be recognized by the structure
of its factors?
\end{question}

\subsection{Jordan-H\" older theorem}\label{jordan-holder}\label{cinco}

In this subsection we focus our discussion on lower composition series. We point
out that, by
duality, analogous results hold for upper composition series. See Remark
\ref{duality}.  

\begin{theorem}\label{schreier}\emph{(Schreier's refinement theorem.)} Let $H$
be a finite dimensional Hopf algebra. Then any two lower
subnormal series of $H$ admit equivalent refinements.
\end{theorem}

\begin{proof} Let $(\mathcal A)$, $(\mathcal B)$ be two lower subnormal series
of $H$, where $$(\mathcal A): \; k = A_{n} \subseteq A_{n-1} \subseteq \dots
\subseteq A_0 = H,\quad (\mathcal B): \; k = B_{m} \subseteq B_{m-1} \subseteq
\dots
\subseteq B_0 = H,$$ $n, m \geq 1$.
For every $0 \leq i \leq n-1$, $0 \leq j \leq m$, let $A_{i, j} =
A_{i+1}(B_j\cap A_i)$.
By Theorem \ref{butterfly} (i), $A_{i, j}$ is a normal Hopf subalgebra of $A_{i,
j-1}$, for all $j = 1, \dots, m$.
In addition, $A_{i, 0} = A_i$, for all $i = 0, \dots, n-1$,  and $A_{i+1}$ is a
normal Hopf subalgebra of $A_{i, j}$, for all $i = 0, \dots, n-1$, $j = 0,
\dots, m$.
Hence we obtain a refinement of $(\mathcal A)$:
\begin{align*}k & \subseteq A_{n-1, m-1} \subseteq A_{n-1,m-2} \subseteq \dots
\subseteq A_{n-1,0} = A_{n-1} \subseteq A_{n-2, m-1}  \subseteq \dots \subseteq
A_{n-2,0} \\
& = A_{n-2} \subseteq \dots \subseteq A_{1, m-1}  \subseteq \dots \subseteq
A_{1,0} =
A_1 \subseteq A_{0, m-1} \subseteq \dots \subseteq A_{0,0} =
H. \end{align*}
Similarly, let $B_{j, i} = B_{j+1}(A_i \cap B_j)$, for all $0 \leq j \leq m-1$,
$0 \leq i \leq n$. Then $B_{j, i+1}$ is a normal Hopf subalgebra of $B_{j, i}$,
and we
obtain a refinement of $(\mathcal B)$:
\begin{align*}k & \subseteq B_{n-1, m-1} \subseteq B_{n-1,m-2} \subseteq \dots
\subseteq B_{n-1,0} = B_{n-1} \subseteq B_{n-2, m-1}  \subseteq \dots \subseteq
B_{n-2,0} \\
& = B_{n-2} \subseteq \dots \subseteq B_{1, m-1} \subseteq \dots \subseteq
B_{1,0} =
B_1 \subseteq B_{0, m-1}  \subseteq \dots \subseteq B_{0,0} =
H.
\end{align*}
There is a bijection between the set of indices in these new subnormal
series, induced by the map $(i, j) \to (j, i)$, $0 \leq i \leq n-1$, $0\leq j
\leq m-1$. Note that $A_{i, 0} = A_{i-1, m}$ and $B_{j, 0} = B_{j-1, m}$, $i,
j\geq 1$. 
The corresponding factors are, respectively
\begin{align*}\dfrac{A_{i, j}}{A_{i, j}A_{i, j+1}^+} & = \dfrac{A_{i+1}(B_j\cap
A_i)}{A_{i+1}(B_j\cap A_i) (A_{i+1}(B_{j+1}\cap A_i))^+}, \quad \textrm{ and 
}\\
\dfrac{B_{j, i}}{B_{j, i}B_{j, i+1}^+} & = \dfrac{B_{j+1}(A_i \cap
B_j)}{B_{j+1}(A_i \cap B_j) (B_{j+1}(A_{i+1} \cap B_j))^+}.\end{align*}
It follows from Theorem \ref{butterfly} (iv) that the corresponding factors are
isomorphic. Therefore the series $(\mathcal A)$ and $(\mathcal B)$ are
equivalent. 
\end{proof}

Since a lower composition series of $H$ admits no proper refinement,  we obtain:

\begin{theorem}\label{jh}\emph{(Jordan-H\"older theorem
for lower subnormal series.)} Any two lower composition series of a finite
dimensional Hopf algebra $H$ are equivalent. \qed \end{theorem}

Theorem \ref{jh} allows us to introduce the following invariants of a finite
dimensional Hopf algebra $H$.

\begin{definition} Let $H$ be a finite dimensional Hopf algebra and let  
$k = H_n \subseteq H_{n-1} \subseteq \dots \subseteq H_1 \subseteq H_0 = H$
be a lower composition series of $H$. The factors of the series  will be called
the \emph{lower composition factors of $H$} are the factors of $H$.
The number $n$ will be called the \emph{lower length of $H$}.
\end{definition}

Upper composition factors and upper length can be defined similarly. 

\begin{remark} Examples \ref{cocommutative} and \ref{dual-group} show
that the lower lengths of a Hopf algebra and its dual may be different. Compare
with Corollary \ref{dual-inv}.
 \end{remark}

\begin{proposition} Let $m$ be the length of $H$ and let $n$, $u$ be its lower
and upper lengths, respectively. Then $n \leq m$ and $u \leq m$. 
\end{proposition}
 
Note that all three lengths coincide if all lower (or upper) factors are simple
Hopf algebras. See Proposition \ref{simple-factors}. 
 
\begin{proof} We shall prove the inequality for the lower length.
This implies the inequality for the upper length in view of Corollary
\ref{dual-inv} and Remark \ref{duality}. 
The proof is by induction on the dimension of $H$. If $\dim H =
1$, there is nothing to prove. 
Suppose that $\dim H > 1$ and let $k = H_n \subseteq \dots \subseteq H_1
\subseteq H_0 = H$ be a lower composition series of $H$. Then $H_1$ is a normal
Hopf subalgebra and $\overline H_0 = H/H{H_1}^+$ is a composition factor
of $H$. It follows from Corollary \ref{additive} that the length $m$ of
$H$ equals $m_1+m'_1$, where $m'_1 \geq 1$ is the length of $\overline H_0$. 
Observe that $k = H_n \subseteq \dots \subseteq H_1$ is a lower composition
series of $H_1$, hence the lower length of $H_1$ equals $n-1$. The inductive
assumption implies that $n-1 \leq m_1$.
Then $m = m_1 + m'_1 \geq n -1 + m'_1 \geq n$. 
\end{proof}

 \begin{example}\label{fl-kernel} Keep the notation in Example \ref{flker-max}.
Suppose $\mathfrak{g}$ is a simple Lie algebra of rank $n$. Let $A =
(a_{ij})_{i,j}$ be the Cartan matrix of $\mathfrak{g}$ and let $D =
(d_1, \dots, d_n)$ be a diagonal matrix with relatively prime entries such that
$DA =AD$.

Let $\mathfrak{j} \cong k \mathbb Z_2^{(n)}$ be the central Hopf
subalgebra of $\mathfrak u$ considered in \cite[p. 27]{notes-ext}. Let us denote
$\overline{\mathfrak u} = \mathfrak{u}/\mathfrak{u}\mathfrak{j}^+$. 
Suppose that $\ell$ and the
determinant of the matrix $(d_ia_{ij})_{ij}$ are relatively prime (which is
always the case if $\mathfrak{g}$ is not of type $A$). Then 
$\overline{\mathfrak u}$ is a simple Hopf algebra
\cite[Appendix]{notes-ext}. 
Then $k \subseteq k \mathbb Z_2 \subseteq \dots \subseteq k \mathbb
Z_2^{(n-1)} \subseteq \mathfrak{j} \subseteq \mathfrak{u}$
is a lower composition series of $\mathfrak{u}$ with factors $k\mathbb{Z}_2$
(with multiplicity $n$) and  $\overline{\mathfrak u}$. Since all lower
composition factors are simple Hopf
algebras, then these are also the composition factors of $\mathfrak u$.
Hence, in this case, the  length of $\mathfrak{u}$  is $n+1$ and it
coincides with its lower and upper lengths.   \end{example}
 
\begin{remark}\label{adm-factors}  Let  \eqref{lowerseries} be a lower
subnormal series of $H$ with factors $\overline{H_i}$. 
Then $H_i$ is isomorphic as a
Hopf algebra to a
bicrossed product $H_i \cong H_{i+1}\#\overline H_i$.
Therefore $H = H_0$ can be obtained from $H_{n-1} = \overline H_{n-1}$
through an iterated sequence of bicrossed products by the factors of the series:
$H \cong ((\dots (\overline H_{n-1} \# \overline H_{n-2})\# \dots) \#
\overline H_0)$.
In view of Remark
\ref{ss-coss}, we get that $H$ is semisimple (respectively, cosemisimple)
if and only if all its lower composition factors  are semisimple (respectively,
cosemisimple).  Compare with Proposition \ref{ss-invfactors}.
\end{remark}

\begin{example}\label{ej-ab} Let $\Gamma$ and $F$ are finite groups. Consider an
\emph{abelian} exact sequence
\begin{equation}\label{ab-exact}k \longrightarrow k^\Gamma \longrightarrow H
\longrightarrow kF \longrightarrow k.
\end{equation}
 It is known that \eqref{ab-exact}
gives rise to mutual actions by permutations $\lhd: \Gamma \times F \to \Gamma$
and $\rhd: \Gamma \times F \to F$ that make $(F, \Gamma)$ into a \emph{matched
pair of groups}. 
Moreover, there exist invertible normalized 2-cocycles $\sigma:F
\times F \to k^\Gamma$ and $\tau:\Gamma \times \Gamma \to k^F$, which are
compatible in an appropriate sense, such that $H \cong k^\Gamma {}^\tau\#_\sigma
kF$ is a bicrossed product. See \cite{masuoka}.

For every $s\in \Gamma$, let $e_g \in k^\Gamma$ be defined by $e_s(t) =
\delta_{s, t}$, $t \in \Gamma$.  Then $(e_s\# x)_{s \in \Gamma, x \in F}$ is a
basis of $k^\Gamma {}^\tau\#_\sigma kF$ and, in this basis, the multiplication,
comultiplication and antipode of $H$ are given, respectively, by the formulas 
$$(e_g \# x)(e_h \# y)  = \delta_{g \lhd x, h}\,
\sigma_g(x, y) e_g \# xy, \quad  \Delta(e_g \# x)  = \sum_{st=g} \tau_x(s, t)\,
e_s \# (t \rhd x) \otimes e_{t}\# x,$$
$$\mathcal S(e_g \# x) = \sigma_{(g\lhd x)^{-1}}((g\rhd x)^{-1}, g\rhd x)^{-1}
\, \tau_x(g^{-1}, g)^{-1} \, e_{(g\lhd x)^{-1}} \# (g\rhd x)^{-1},$$
for all $g,h\in \Gamma$, $x, y\in F$, where $\sigma_g(x, y) =
\sigma(x, y)(g)$ and $\tau_x(g, h) = \tau(g, h)(x)$.

Furthermore, the exact sequence \eqref{ab-exact} is equivalent to the sequence
$k \longrightarrow k^\Gamma \overset{i}\longrightarrow k^\Gamma \,
{}^{\tau}\#_{\sigma}kF
\overset{\pi}\longrightarrow kF \longrightarrow k$, where $i(e_g) = e_g \otimes
1$, $g \in \Gamma$, and
$\pi = \epsilon \otimes \id$.

 A subnormal series $\{e\} = F_m \subseteq \dots \subseteq F_0 = F$ of
$F$ will be called a \emph{$\Gamma$-subnormal series} if every term $F_i$ is a
$\Gamma$-stable subgroup of $F$. By a \emph{$\Gamma$-composition series} of $F$
we shall understand a $\Gamma$-subnormal series which does not admit a proper
refinement consisting of $\Gamma$-stable subgroups. 

\begin{proposition}\label{cs-abext} Let $\{e\} = \Gamma_n \subseteq \dots
\subseteq \Gamma_0 =
\Gamma$ be a principal series of $\Gamma$ and let 
$\{e\} = F_m \subseteq \dots \subseteq F_0 = F$ be a $\Gamma$-composition series
of $F$. Then 
\begin{equation}\label{cs-ab}k \subseteq k^{\Gamma/\Gamma_1} \subseteq \dots
\subseteq  k^\Gamma \subseteq 
k^\Gamma {}^\tau\#_\sigma kF_{m-1} \subseteq \dots \subseteq k^\Gamma
{}^\tau\#_\sigma kF_1 \subseteq H
\end{equation}
is a lower composition series of $H$. 
\end{proposition}

\begin{proof} Since the subgroups $F_i$ are $\Gamma$-stable, then
$(\Gamma, F_i)$ is a matched pair by restriction and $k^\Gamma {}^\tau\#_\sigma
kF_i$ can be identified with a Hopf subalgebra of $H$. Moreover,
since $F_i$ is normal in $F_{i-1}$, then $k^\Gamma
{}^\tau\#_\sigma kF_i$ is a normal Hopf subalgebra of $k^\Gamma {}^\tau\#_\sigma
kF_{i-1}$, for all $i =1, \dots, m$.
It was already observed (Example \ref{dual-group}) that  
$k^{\Gamma/\Gamma_i}$
is a maximal normal Hopf subalgebra of $k^{\Gamma/\Gamma_{i+1}}$, for all $i =
0, \dots, n-1$. Let $1 \leq i \leq m$, and suppose that $L$ is a normal Hopf
subalgebra of $k^\Gamma {}^\tau\#_\sigma kF_{i-1}$ such that $k^\Gamma
{}^\tau\#_\sigma kF_i \subsetneq L \subseteq k^\Gamma {}^\tau\#_\sigma
kF_{i-1}$. 
Under the canonical projection $\pi = \epsilon \otimes \id: H \to kF$, we
have $\pi(L) \subseteq \pi(k^\Gamma {}^\tau\#_\sigma kF_{i-1}) = kF_{i-1}$.
Let $\tilde F$ be a subgroup of $F_{i-1}$ such that $\pi(L) = k\tilde F$. Then
$\dim L = \dim L^{\co \pi} \dim \pi(L) = |\Gamma| |\tilde F|$.
Since $L$ is a normal in $k^\Gamma {}^\tau\#_\sigma kF_{i-1}$
and $k^\Gamma \subseteq L$, then $\tilde F$ is
a $\Gamma$-stable normal subgroup of $F_{i-1}$. 
Notice that $kF_i = \pi (k^\Gamma {}^\tau\#_\sigma kF_i) \subseteq k\tilde F$,
hence $F_i \subseteq \tilde F$ and $F_i \neq \tilde F$, because
$|\Gamma| |\tilde F| = \dim L  > \dim k^\Gamma {}^\tau\#_\sigma kF_i = |\Gamma|
|F_i|$.  
By maximality of the series $(F_i)_i$, we get $\tilde F =F_{i-1}$. As before,
this implies that $\dim L =
\dim k^\Gamma {}^\tau\#_\sigma kF_{i-1}$ and then $L = k^\Gamma {}^\tau\#_\sigma
kF_{i-1}$.  
This shows that $k^\Gamma {}^\tau\#_\sigma kF_i$ is a maximal normal Hopf
subalgebra of $k^\Gamma {}^\tau\#_\sigma kF_{i-1}$. Then 
\eqref{cs-ab} is a lower composition series of $H$, as claimed.
\end{proof}

\begin{remark} Recall that the length of a finite
dimensional Hopf algebra is additive with respect to exact sequences; see
Corollary \ref{additive}.
In contrast with this situation,  
Proposition \ref{cs-abext} provides nontrivial examples of exact sequences of
Hopf algebras 
$k \longrightarrow H' \longrightarrow H \longrightarrow H'' \longrightarrow
k$,
such that the lower length of $H$ is different from the sum of the lower lengths
of
$H'$ and
$H''$.  Moreover, a lower composition factor of $H$ needs not be a lower
composition factor of $H'$ or of $H''$.
 \end{remark}
\end{example}

\begin{example} Suppose that \eqref{ab-exact} is a \emph{cocentral} exact
sequence 
(equivalently, that the action $\rhd: \Gamma \times F \to F$ is trivial). Then a
$\Gamma$-composition series of $F$ is the same as a composition series of $F$.
Hence the lower composition factors of $H$ are
the principal factors of $\Gamma$ and the composition factors of $F$.

An example of a cocentral abelian extension is given by the Drinfeld double
$D(G)$ (see Example \ref{ddoble}). As a
Hopf algebra, $D(G)$ is a bicrossed product $D(G) = k^G \# kG$, with respect to 
the trivial action $\rhd: G \times G \to G$ and the adjoint action $\lhd: G
\times G
\to G$, where $\sigma = 1$, $\tau = 1$.
Therefore the lower composition factors of $D(G)$ are the dual group algebras of
the principal factors
of $G$ and the group algebras of the composition factors of $G$. In
particular, the lower length of $D(G)$ equals the sum of the principal length
and the length of $G$.
On the other hand, the dual Hopf algebra $D(G)^*$ is a bicrossed product $k^G\#
kG$, with respect to the adjoint action $\rhd: G \times G \to G$ and the trivial
action $\lhd: G \times G \to G$. In this case, a $G$-composition
series of $G$ is the same as a normal series of $G$. Hence the upper
composition factors of $D(G)$ are the dual group
algebras and the group algebras of the principal factors of $G$; see Remark
\ref{duality}. In particular, if $G$ is any finite group whose composition length is
different from its principal length, then $D(G)$ provides an
example of a Hopf algebra for which the lower length, the upper length and the 
length are pairwise distinct.
\end{example}

\subsection{On lower composition series of normal right coideal
subalgebras}\label{ncs}

Let $H$ be a finite dimensional Hopf algebra over $k$. The underlying algebra of
 $H$ is
an algebra in the category $\mathcal{YD}^H_H$ of Yetter-Drinfeld modules over
$H$
with respect to the right adjoint
action and the right coaction given by the comultiplication. 
We shall use the notation $\mathcal H$ to indicate this structure.
Observe that the subalgebras of $\mathcal H$ in $\mathcal{YD}^H_H$ are exactly
the right
normal right coideal subalgebras of $H$.
By \cite{skryabin} $H$ is
free over its coideal subalgebras. Hence, by
\cite[Theorem 3.2]{takeuchi}, the maps $K \mapsto H/HK^+$,  $\overline H
\mapsto  ^{\co \overline
H}\!H$, determine inverse bijections between the following two sets:
\begin{itemize}\item[(a)] Right normal right coideal subalgebras $K
\subseteq H$. \item[(b)] Quotient Hopf algebras $H \to
\overline H$. \end{itemize}
Since every tensor subcategory of $\Rep H$ is of the form $\Rep \overline H$,
for some quotient Hopf algebra $\overline H$, it follows that the map $A \mapsto
\Rep H/HA^+$ determines an
anti-iso\-morphism between the following lattices:
\begin{itemize}\item[(a')] Yetter-Drinfeld subalgebras $A$ of $\mathcal H$.
\item[(b')] Tensor  subcategories of $\Rep H$.\end{itemize}                     
                      
Let us consider a maximal series  of algebras in $\mathcal{YD}^H_H$
(\textit{i.e.}, normal right coideal subalgebras of $H$):
$k = A_{n} \subseteq A_{n-1} \subseteq \dots
\subseteq A_1 \subseteq A_0 =
\mathcal H$. 
This corresponds to a maximal series
of tensor subcategories $\vect_k = \Rep H/HA_0^+ 
\subseteq \Rep H/HA_1^+ \subseteq \dots \subseteq \Rep H/HA_n^+ = \Rep H$.
The following example shows that these series fail to admit a Jordan-H\" older
theorem.

\begin{example}
Let $G$ be a finite group and consider the Hopf algebra $H =
k^G$. Every right coideal subalgebra of $k^G$ is of the form $k^{H \backslash
G}$,  where $H$ is a subgroup of $G$ and $H \backslash G = \{Hg:\, g \in G\}$
is the space of right cosets of $G$ modulo $H$.
Then a maximal series of right coideal subalgebras corresponds to a maximal
series of (not necessarily normal) subgroups $\{e\}= G_0 \subsetneq G_1
\subsetneq \dots \subsetneq G_n = G$. 

Take for instance $G = \A_5$. Then the following are maximal series of
subgroups:
$\{e\} \subsetneq \langle   (123) \rangle
\subsetneq \langle   (123), (12)(45) \rangle \subsetneq \A_5$,
and $\{e\} \subsetneq \langle   (12)(34) \rangle \subsetneq  \langle   (12)(34),
(14)(23) \rangle$ $\subsetneq \langle   (123),
(12)(34) \rangle \subsetneq  \A_5$.
In the first case, we have $\langle   (123),
(12)(45) \rangle \cong \mathbb S_3$, while in the second case, $\langle   (123),
(12)(34)
\rangle \cong \A_4$ and $\langle   (12)(34), (14)(23) \rangle \cong \mathbb
Z_2\times \mathbb Z_2$.
We observe in this example that not even the length of a maximal series
of normal right coideal subalgebras is a well-defined invariant of $H$.
\end{example}

\bibliographystyle{amsalpha}

\end{document}